\documentclass[11pt,letterpaper]{amsart}
\usepackage{amsmath,amssymb,amsthm,enumitem}

\usepackage{multirow}
\usepackage{tikz}
\usepackage{amsfonts}
\usepackage{caption}
\usepackage{color}

\usetikzlibrary{matrix,arrows,decorations.pathmorphing}

\DeclareMathOperator{\Tr}{Tr}

\DeclareMathOperator{\Gal}{Gal}
\DeclareMathOperator{\Ima}{Im}
\DeclareMathOperator{\Ker}{Ker}

\DeclareMathOperator{\id}{id}

\renewcommand{\phi}{\varphi}

\newtheorem{theorem}{Theorem}[section]
\newtheorem{proposition}[theorem]{Proposition}
\newtheorem{lemma}[theorem]{Lemma}

\theoremstyle{definition}

\newtheorem{definition}[theorem]{Definition}

\def\cqfd{
{\hfill
\kern 6pt\penalty 500
\raise -1pt\hbox{\vrule\vbox to 5pt{\hrule width 4pt
\vfill\hrule}\vrule}}
\break}

\def\Gal{\mathop{\rm Gal}\nolimits}

\def\Tr{\mathop{\rm Tr}\nolimits}

\font\tengoth=eufm10
\font\sevengoth=eufm7
\font\fivegoth=eufm5
\newfam\gothfam
\textfont\gothfam=\tengoth\scriptfont\gothfam=\sevengoth\scriptscriptfont\gothfam=\fivegoth
\def\goth{\fam\gothfam\tengoth}

\def\Sgoth{{\goth S }}

\title[Differential uniformity and second order derivatives]{Differential uniformity and second order derivatives for generic polynomials}
\author[Yves Aubry]{Yves Aubry}
\address[Yves Aubry]{Institut de Math\'ematiques de Toulon, Universit\'e de Toulon, France}
\address[Yves Aubry]{Aix-Marseille Univ, CNRS, Centrale Marseille, I2M, Marseille, France}
\email{yves.aubry@univ-tln.fr}

\author[Fabien Herbaut]{Fabien Herbaut}
\address[Fabien Herbaut]{Institut de Math\'ematiques de Toulon, Universit\'e de Toulon, France}
\address[Fabien Herbaut]{ESPE Nice-Toulon, Universit\'e Nice Sophia Antipolis, France}
\email{fabien.herbaut@unice.fr}

\begin{document} 

\begin{abstract}
{For any polynomial  $f$ of ${\mathbb F}_{2^n}[x]$ we introduce 
the following characteristic of the distribution of its 
second order derivative,
which extends the differential uniformity notion:
$$\delta^2(f):=\max_{\substack{
\alpha \in {\mathbb F}_{2^n}^{\ast} ,\alpha' \in {\mathbb F}_{2^n}^{\ast} ,\beta \in {\mathbb F}_{2^n} \\
\alpha\not=\alpha'}} 
\sharp\{x\in{\mathbb F}_{2^n} \mid D_{\alpha,\alpha'}^2f(x)=\beta\}$$
where 
$D_{\alpha,\alpha'}^2f(x):=D_{\alpha'}(D_{\alpha}f(x))=f(x)+f(x+\alpha)+f(x+\alpha')+f(x+\alpha+\alpha')$ is  the second order derivative.
Our purpose is to prove a density theorem relative to this quantity,
which is an analogue of a density theorem proved by Voloch for the differential uniformity.}
\end{abstract}

\maketitle

\noindent
{\bf Keywords}: Differential uniformity, Galois closure of a map, Chebotarev density theorem.

\medskip
\noindent
{\bf Mathematics Subject Classification}: 14G50, 11T71, 94A60.

\section{Introduction}
For any polynomial  $f\in {\mathbb F}_{q}[x]$ where $q=2^n$, and for $\alpha\in{\mathbb F}^{\ast}_{q}$,  the derivative of $f$ with respect to $\alpha$ is  the polynomial 
$D_{\alpha}f(x)=f(x+\alpha)+f(x).$
The differential uniformity $\delta(f)$ of $f$ introduced by Nyberg in \cite{Nyberg} is then defined by
$$\delta(f):=\max_{(\alpha,\beta)\in{\mathbb F}_q^{\ast}\times{\mathbb F}_q}\sharp\{x\in{\mathbb F}_{q} \mid D_{\alpha}f(x)=\beta\}.$$
To stand against differential cryptanalysis, one wants to have a small differential uniformity (ideally equal to 2).
Voloch proved  that most polynomials $f$ of ${\mathbb F}_{q}[x]$ 
of degree $m\equiv 0,3\pmod 4$ have a differential uniformity equal to $m-1$ or $m-2$
(Theorem 1 in \cite{Voloch}).

When studying differential cryptanalysis, 
Lai introduced in \cite{Lai} the notion of higher order derivatives.
The higher order derivatives are defined recursively by 
$D_{\alpha_1, \ldots , \alpha_{i+1}} f = D_{\alpha_1, \ldots , \alpha_{i}}(D_{\alpha_{i+1} } f )$, and a
new design principle is given in \cite{Lai}:
\textit{"For each small i, the nontrivial i-th derivatives of function 
should take on each possible value roughly uniform"}.
After considering the differential uniformity,
it seems natural to investigate the number of solutions of the equation
$D_{\alpha_1,\alpha_2} f (x) = \beta$, that is of the equation
$$f(x)+f(x+\alpha_1)+f(x+\alpha_2)+f(x+\alpha_1+\alpha_2)=\beta$$
 and thus to consider the second order differential uniformity of $f$ over ${\mathbb F}_q$:
$$\delta^2(f):=
\max_{\substack{
\alpha \in {\mathbb F}_q^{\ast} ,\alpha' \in {\mathbb F}_q^{\ast} ,\beta \in {\mathbb F}_q \\
\alpha\not=\alpha'}} 
\sharp\{x\in{\mathbb F}_{q} \mid D_{\alpha,\alpha'}^2f(x)=\beta\}.$$

\bigskip
For example, the inversion mapping from ${\mathbb F}_q$ to itself which sends $x$ to $x^{-1}$ if $x\not=0$ and 0 to 0 (and which corresponds to the polynomial $f(x)=x^{q-2}$) has a differential uniformity $\delta(f)=2$ for $n$ odd and $\delta(f)=4$ for $n$  even (see \cite{Nyberg}).
We will prove in Section \ref{inverse} that it has a  second order differential uniformity $\delta^2(f)=8$ for any $n\geqslant 6$.

\bigskip

The purpose of the paper is to prove that, as Voloch proved it for the differential uniformity, most polynomials $f$ have 
a maximal $\delta^2(f)$.
More precisely, we prove (Theorem \ref{maintheorem}) that:
for a given integer $m \geqslant 7$
 such that $m\equiv 0 \pmod 8$ 
(respectively $m \equiv 1,2,7\pmod 8$),
and with $\delta_0=m-4$ (respectively  $\delta_0=m-5, 
m-6, m-3$) we have
$$\lim_{n\rightarrow \infty}
\frac{\sharp\{f\in{\mathbb F}_{2^n}[x] \mid \deg(f)=m,\ \delta^2(f)=\delta_0\}}
{\sharp\{f\in{\mathbb F}_{2^n}[x] \mid \deg(f)=m\}}=1.$$

We follow and generalize  the ideas of Voloch in \cite{Voloch}. Let us present the strategy.
\medskip

- In Section \ref{section_nabla}, we associate  to any integer $m$  an integer $d$ depending on the congruence of $m$ modulo 4 (Definition \ref{definition_de_d}). Then,
if $\alpha$ and $\alpha'$ are two disctinct elements of ${\mathbb F}_{q}^{\ast}$,
we associate (Proposition \ref{polynomeg})
to any polynomial $f\in {\mathbb F}_{q}[x]$ of degree $m$
 a polynomial $L_{\alpha,\alpha'}(f)$
(which will be sometimes denoted by $g$ for simplicity)
of degree less than or equal to $d$ 
such that:
$$D_{\alpha,\alpha'}^2f(x)=g\bigl(x(x+\alpha)(x+\alpha')(x+\alpha+\alpha')\bigr).$$

\medskip

- In Section \ref{monodromy}, we determine the geometric and the arithmetic monodromy groups of $L_{\alpha,\alpha'}(f)$ 
when this polynomial is Morse  (Proposition \ref{Monodromy}). For $\alpha$ and $\alpha'$ fixed, 
we give an upper bound depending only on $m$ and $q$ for the number  of polynomials $f$ of $\mathbb{F}_q[x]$ of degree at most $m$
such that $L_{\alpha, \alpha'}(f)$ is non-Morse (Proposition \ref{proposition:bounded_nonMorse}).

\medskip

- Section \ref{section:Geometric_and_arithmetic_mondormy_groups} is devoted to the study of the monodromy groups of 
$D_{\alpha,\alpha'}^2f$.
In order to apply the Chebotarev's density theorem (Theorem \ref{Chebotarev}) we look for a condition of regularity, that is a condition for ${\mathbb F}_{q}$ to be algebraically closed in 
 the Galois closure of the polynomial $D^2_{\alpha,\alpha'}f(x)$ (Proposition \ref{regular}).

\medskip

- In Section \ref{section_chebotarev}, we use the Chebotarev theorem to prove that (Proposition \ref{application_Chebotarev})  for $q$ sufficiently large and under the regularity hypothesis  the polynomial 
$D^2_{\alpha,\alpha'}f(x)+\beta$ totally
splits in $ \mathbb{F}_q[x]$.

\medskip

 - In Section \ref{section_good}, we show that we can choose a finite set of couples $(\alpha_i, \alpha_i')$ such that  
 most  polynomials  $f\in {\mathbb F}_{q}[x]$ of degree $m$ satisfy
 the above regularity  condition  (Proposition \ref{covers}).
 
 \medskip
 
 - Finally,  Section \ref{section_main_theorem} is devoted to the
 statement and the proof of  the main theorem (Theorem \ref{maintheorem}).
 
\bigskip

To fix  notation, 
throughout the whole paper we consider $n$ a non-negative integer 
and $q=2^n$. We denote by  ${\mathbb F}_q$  the finite field with $q$ elements, by
${\mathbb F}_q[x]$ the ring of polynomials in one variable over ${\mathbb F}_q$ and by
${\mathbb F}_q[x]_m$ the ${\mathbb F}_q$-vector space of polynomials of ${\mathbb F}_q[x]$ of degree at most $m$.
We will often consider a polynomial $f\in{\mathbb F}_q[x]$ of degree $m$, an element $\beta$ of ${\mathbb F}_q$ and distincts elements $\alpha$ and $\alpha'$ in ${\mathbb F}_q^{\ast}$.

\section{The associated polynomial $L_{\alpha, \alpha'}(f)$}\label{section_nabla}
The derivative of a polynomial $f\in{\mathbb F}_q[x]$
along $\alpha\in{\mathbb F}^{\ast}_q$ is defined by
$$D_{\alpha}f(x)=f(x)+f(x+\alpha)$$
and its second derivative along $(\alpha,\alpha') \in {\mathbb F}_q^2$ is defined by 
$$D^2_{\alpha,\alpha'}f(x)=D_{\alpha} \left( D_{\alpha' }   f \right)  (x)  = f(x)+f(x+\alpha)+f(x+\alpha')+f(x+\alpha+\alpha').$$
Actually $D^2_{\alpha,\alpha'}f$ depends only on the
$\mathbb{F}_2$-vector space generated by $\alpha$ and $\alpha'$.
If $f\in \mathbb{F}_q[x]$ is of odd degree $m$, then for any $\alpha\in{\mathbb F}_q^{\ast}$
the degree of $D_{\alpha}f$ is $m-1$. On the other hand, 
if $m$ is even then  the degree of $D_{\alpha}f$ is 
less than or equal to $m-2$.
Consequently,  if $\alpha'\in{\mathbb F}_q^{\ast}$ 
we obtain that
the degree of 
$D^2_{\alpha,\alpha'}f$ is less than or equal to $m-3$ when $m$ is odd, 
and less than or equal to $m-4$ otherwise.
To any  integer $m \geqslant 7$  we associate the following integer $d=d(m)$ 
(we will often omit the dependance in $m$).

\begin{definition}\label{definition_de_d}
 Let $m$ be an integer greater or equal to $7$. 
If $m\equiv 0\pmod 4$ we set $d=\frac{m-4}{4}$,
if $m\equiv 1\pmod 4$ we set $d=\frac{m-5}{4}$, if $m\equiv 2\pmod 4$ we set $d=\frac{m-6}{4}$ and
if $m\equiv 3\pmod 4$ we set $d=\frac{m-3}{4}$.
\end{definition}

We sum up the situation in the following table.
\begin{center}
\begin{tabular}{|c|c|c|} 
\hline
$m \pmod 4$ & $\deg D^2_{\alpha,\alpha'} f$ & $d$   \\
\hline 
$0$            & $\leqslant m-4$                                       & $\frac{m-4}{4}$   \\
1            & $\leqslant m-3$                                       & $\frac{m-5}{4} $  \\
2            & $\leqslant m-4 $                                      & $\frac{m-6}{4} $ \\
3           & $\leqslant m-3 $                                       & $\frac{m-3}{4} $ \\
\hline
\end{tabular}
\captionof{table}{Definition of $d$}
\label{table:definition_of_d}
\end{center}
\begin{proposition}\label{polynomeg}
Let $\alpha, \alpha'\in{\mathbb F}_q^{\ast}$ such that $\alpha\not=\alpha'$ and
let $f\in{\mathbb F}_q[x]$ be a polynomial of degree $m$. 
There exists a unique polynomial $g\in{\mathbb F}_q[x]$ 
of degree less than or equal to
 $d$  such that
$$D_{\alpha,\alpha'}^2f(x)=g\bigl(x(x+\alpha)(x+\alpha')(x+\alpha+\alpha')\bigr).$$
Moreover, the map 
$$
\begin{matrix}
L_{\alpha,\alpha'}&:{\mathbb F}_q[x] & \longrightarrow & {\mathbb F}_q[x]\\
& f & \longmapsto & g
\end{matrix}
$$
is linear and 
$L_{\alpha , \alpha'}({\mathbb F}_q[x]_m)={\mathbb F}_q[x]_d$.
\end{proposition}

\begin{proof}
Fix $f$ a polynomial of degree $m$ and $\alpha, \alpha'\in{\mathbb F}_q^{\ast}$ such that $\alpha\not=\alpha'$.
Let us first prove the
existence of $g$.
If $D^2_{\alpha,\alpha'}f$ is the zero polynomial then $g=0$ is suitable. Suppose now that 
$D^2_{\alpha,\alpha'}f$ is non-zero and
 set $c$ for its leading coefficient 
and $\Lambda_{k}$ the set of its roots 
of multiplicity $k$ in an algebraic closure $\overline{\mathbb{F}}_q$ of ${\mathbb F}_q$.
As $x \mapsto x + \alpha$ and $x \mapsto x + \alpha'$ are
two involutions
of each set $\Lambda_{k}$, 
 there exists $ \Lambda_{k}' \subset \Lambda_{k}$ such that:
$$ D^2_{\alpha,\alpha'}f(x) 
 = c \prod_{\scriptstyle k \geqslant 1} 
 \prod_{\scriptstyle \lambda \in   \Lambda_k' } 
{ 
(x+\lambda)^k(x+\lambda + \alpha )^k (x+\lambda+\alpha' )^k(x+\lambda+\alpha+\alpha')^k }.$$
Hence
$$ D^2_{\alpha,\alpha'}f(x)  = 
 c \prod_{\scriptstyle k \geqslant 1} 
\prod_{\scriptstyle \lambda \in  \Lambda_k'  }  
{ \scriptstyle 
\big{(} x^4+(\alpha^2 + \alpha'^2+\alpha \alpha')x^2 + (\alpha^2 \alpha' + \alpha \alpha'^2)x   
+ \lambda^4+(\alpha^2 + \alpha'^2+\alpha \alpha')\lambda^2 + (\alpha^2 \alpha' + \alpha \alpha'^2)\lambda 
 \big{)}^k }$$
 $$= 
  c \prod_{\scriptstyle k \geqslant 1} 
\prod_{\scriptstyle \lambda \in  \Lambda_k'  }  
{ 
{ \scriptstyle
\left(
x(x+ \alpha)(x+ \alpha')(x+ \alpha+ \alpha') +
\lambda(\lambda+ \alpha)(\lambda+ \alpha')(\lambda+ \alpha+ \alpha')  
\right)^k }}.$$
Then the polynomial $g$ defined by 
\[ 
g(x) =  c \prod_{\scriptstyle k \geqslant 1} 
\prod_{\scriptstyle \lambda \in  \Lambda_k'  }   
\left(x+ \lambda (\lambda + \alpha)(\lambda+ \alpha')( \lambda+\alpha+\alpha') \right)^k
\] 
satisfies $g \left( x ( x+\alpha)( x+\alpha')( x+\alpha+\alpha') \right)
 = D^2_{\alpha,\alpha'}f(x)$
and has degree at most 
$d$.
To prove that $g \in {\mathbb F}_q [x]$, 
one can quote linear algebra arguments. 
Actually, solving
$g(x(x+\alpha)(x+\alpha')(x+\alpha+\alpha'))=D^2_{\alpha,\alpha'}f(x)$ 
amounts to solving an affine equation with coefficients in ${\mathbb F}_q$ and
we have already proven that this equation admits solutions with coefficients
in $\overline{\mathbb{F}}_q$. 
As the existence of solutions of such affine equations does not 
depend of the extension field considered, 
we have
solutions with coefficients
in ${\mathbb F}_q$.
The uniqueness is a consequence of the linearity of composition.
To prove the surjectivity of $L_{\alpha,\alpha'}$, we will determine
 the dimension of its kernel and apply
the rank-nullity theorem.
Note that  $f \in \Ker L_{\alpha, \alpha'} $
if and only if $D^2_{\alpha, \alpha'} f = 0$.
But $D^2_{\alpha, \alpha'} f = D_{\alpha} D_{\alpha'} f  $,
so $\Ker L_{\alpha, \alpha'} = D_{\alpha'}^{-1} \left( \Ker D_{\alpha} \right)$.
Classical linear algebra  properties give the equality
$\dim \Ker L_{\alpha, \alpha'} = 
\dim \left( \Ima D_{\alpha'} \cap \Ker D_{\alpha}  \right) 
+ \dim \Ker \left( D_{\alpha'} \right)$.
We conclude  
separating cases according to the congruence of $m$ modulo $4$
and
using Lemma 
 \ref{lemma:ker_image_nabla}.
\end{proof}

For simplicity of notation
we continue to write $D_{\alpha}$ for the restriction of $D_{\alpha}$
to the subspace of polynomials of degree less than or equal to $m$.
We also use the notations
$\lfloor a \rfloor$ for the 
greatest integer less than or equal to $a$ 
and $\lceil a \rceil$ for 
the least integer greater than or equal to $a$.
\begin{lemma}\label{lemma:ker_image_nabla}
Let $\alpha$ and $\alpha'$ be two distinct elements in $\mathbb{F}_q^{\ast}$.
We have: \\
(i) \ $\Ker D_{\alpha} = \{ 
h \left( x(x+ \alpha)  \right) \ | \ \deg(h) \leqslant  \lfloor m/2 \rfloor \}$.\\
(ii) \ $\Ima D_{\alpha} = \{ 
h \left( x(x+ \alpha)  \right) \ | \ \deg(h) \leqslant \lceil m/2 \rceil -1 \}$.\\
(iii) \ If $m$ is
odd, then 
$$\Ima D_{\alpha'} \cap \Ker D_{\alpha} = \{ 
h \left( x(x+ \alpha) (x+ \alpha')(x+ \alpha + \alpha') \right) \ | \ \deg(h) \leqslant  m/4 \}.$$ 
(iv) \ If $m$ is
even, then 
$$\Ima D_{\alpha'} \cap \Ker D_{\alpha} = \{ 
h \left( x(x+ \alpha) (x+ \alpha')(x+ \alpha + \alpha') \right) \ | \ \deg(h) \leqslant  (m-2)/4 \}.$$
\end{lemma}
\begin{proof}
If $D_{\alpha} f=0$ then $f(x)=f(x+ \alpha)$.
The map $x \mapsto x + \alpha$ induces a bijection onto 
the sets of the roots  of $f$ of same multiplicity.
Using the method of the proof of Proposition \ref{polynomeg} we prove {\it (i)}.
We deduce {\it (ii)} proving an easy inclusion and the rank-nullity theorem.
To prove {\it (iii)}, use that if $m$ is odd 
then $\Ima D_{\alpha'}= \Ker   D_{\alpha'}$ by {\it (i)} and {\it (ii)}.
Suppose that $f \in \Ker  D_{\alpha'} \cap \Ker  D_{\alpha}$.
If $x_0$ is a root of $f$ of multiplicity $k$, so are $x_0+ \alpha$,
$x_0+ \alpha'$ and $x_0+ \alpha + \alpha'$, and we can use the method 
of the proof of Proposition \ref{polynomeg} .
We prove {\it (iv)} using the same method and noticing that the intersection
$\Ima D_{\alpha'} \cap \Ker D_{\alpha}$ consists of the polynomials 
of $\Ker  D_{\alpha'} \cap \Ker  D_{\alpha} $ of degree
less than or equal to $m-2$.
\end{proof}

\section{Monodromy groups and Morse polynomials}\label{monodromy}
Let $g\in{\mathbb F}_q[x]$ be a polynomial of degree $d$. 
We consider
the  field extension ${\mathbb F}_q(u)/{\mathbb F}_q(t)$ corresponding to the polynomial $g$ where $t$ is transcendental over ${\mathbb F}_q$ i.e. with $u$ such that  $g(u)-t=0$.
Denote by $F$ the Galois closure  of ${\mathbb F}_q(u)/{\mathbb F}_q(t)$, i.e. $F$ is the splitting field of
$g(x)-t$ over ${\mathbb F}_q(t)$.
The Galois group $\Gal(F/{\mathbb F}_q(t))$ is called the arithmetic monodromy group of $g$.
Let ${\mathbb F}_{q}^F$ be the algebraic closure of ${\mathbb F}_q$ in $F$. Then the Galois group 
$\Gal(F/{\mathbb F}_{q}^F(t))$ 
is a normal 
subgroup of $\Gal(F/{\mathbb F}_q(t))$ called the geometric monodromy group of $g$.
\bigskip

The polynomial $g$ is said to be Morse (see \cite{Serre} p. 39)  if  $g$, viewed 
as a ramified covering $ g \ : \ {\mathbb P}^1  \longrightarrow  {\mathbb P}^1$ of degree $d$,
 is such that above each affine branch point 
there is only one ramification point and the ramification index of such points is 2.
In even characteristic, this notion has to be precised:
following Geyer in the Appendix of \cite{JardenRazon}, the polynomial $g$ is said to be Morse if the three following conditions hold:

\begin{enumerate}[label=\alph*)]

\item  $g'(\tau)=0$ implies that $g^{[2]}(\tau)=0$ where $g^{[2]}$ is the second Hasse-Schmidt derivative,
\item  $g'(\tau)=g'(\eta)=0$ and $g(\tau)=g(\eta)$ imply $\tau=\eta$,
\item  the degree of $g$ is not divisible by the characteristic of ${\mathbb F}_q$.
\end{enumerate}
\bigskip

For Morse polynomials $g$, the  general form of the Hilbert theorem given by Serre in Theorem 4.4.5 of \cite{Serre} adapted to the even characteristic in Proposition 4.2. in the Appendix by Geyer of \cite{JardenRazon} implies that the 
geometric monodromy group $\Gal(F/{\mathbb F}_{q}^F(t))$
is the symmetric group
$\Sgoth_d$. Moreover, it is a subgroup of the arithmetic monodromy group $\Gal(F/{\mathbb F}_q(t))$  and this last group is also contained in $\Sgoth_d$, hence they coincide.

\bigskip

Now let us return to our situation. Let $\alpha, \alpha'$ be two distincts elements of ${\mathbb F}_q^{\ast}$. Let $m$ be an integer and $d=d(m)$ defined 
 in Table \ref{table:definition_of_d}.
Let $f\in{\mathbb F}_q[x]$ be a polynomial of degree $m$. 
Let us consider the polynomial 
$g:=L_{\alpha,\alpha'}(f)\in{\mathbb F}_q[x]$ 
of degree 
 $\leqslant d$  such that
 $$g\bigl(x(x+\alpha)(x+\alpha')(x+\alpha+\alpha')\bigr)=D_{\alpha,\alpha'}^2f(x)$$
whose existence follows from Proposition \ref{polynomeg}.

\begin{proposition}\label{Monodromy}
If $f$ is a  polynomial of degree $m$ such that the polynomial 
$L_{\alpha,\alpha'}(f)$ is of degree exactly $d$ and is Morse
 then the geometric monodromy group, 
 and then also the arithmetic monodromy group of the polynomial $L_{\alpha,\alpha'}(f)$  is the symmetric group
$\Sgoth_d$. Hence the extension $F/{\mathbb F}_q(t)$ is regular i.e. ${\mathbb F}_{q}^F={\mathbb F}_q$.
\end{proposition}

\begin{proof}
By the previous paragraph we have that the geometric and the arithmetic monodromy groups coincide, which gives the regularity property.
\end{proof}

Note that if $L_{\alpha,\alpha'}(f)$ is of degree exactly $d$ and is Morse then  Condition $(c)$ says that  $d$ must be odd. This is equivalent to say that   $m\equiv 0, 1, 2$ or $7$ $\pmod 8$.

\bigskip

Now we give a lower bound for the number of polynomials $f$ 
such that $L_{\alpha, \alpha'}(f)$ is Morse.

\begin{proposition}\label{proposition:bounded_nonMorse}
Let $m \geqslant 7$ such that $m\equiv 0, 1, 2$ or $7$ $\pmod 8$
and $d$ as defined in Definition \ref{definition_de_d}.
There exists an integer $\tilde{d} >0$ depending only on $d$ such that
for any couple $(\alpha, \alpha')$ of distinct elements of $ \mathbb{F}_{q}^{\ast}$
the number of polynomials $f$ of $\mathbb{F}_q[x]$ of degree at most $m$
such that $L_{\alpha, \alpha'}(f)$ is non-Morse
is bounded by $\tilde{d} q^{m}$.
\end{proposition}
\begin{proof}
The loci of non-Morse polynomials 
$g=\sum_{j=0}^d  b_{d-j} x^j$
of $\mathbb{F}_q[x]$ of degree $d$ is a Zariski-closed subset  of the $(d+1)$-dimensional  affine space
with coordinates $b_0,\ldots, b_d$
 given by Geyer in
Proposition 4.3 of the Appendix of \cite{JardenRazon}.  
Indeed, the above condition (a)  means that $g'$ and $g^{[2]}$ have no common root, i.e. the resultant $R(b_0,\ldots,b_d)$ of the polynomials $g'$ and $g^{[2]}$ is non-zero.
Condition (b) above means that the product
$$\Pi=\prod_{i\not=j}(g(\eta_i)-g(\eta_j))$$
where $\eta_i$ are the roots of $g$ does not vanish.
By the theorem on symmetric functions, $\Pi=\Pi(b_0,\ldots,b_d)$ is a polynomial in the coefficients of $g$.

Finally the polynomials $f=\sum_{j=0}^m a_j x^{m-j}$
 such that
$L_{\alpha, \alpha'}(f)$ is non-Morse are those such that
$$R\circ L_{\alpha, \alpha'} (a_0, \ldots , a_m)=0 \ \ {\rm or}\ \ \Pi \circ L_{\alpha, \alpha'} (a_0, \ldots , a_m)=0.$$

The polynomials $R$ and $\Pi$ are proven to be non-zero in Geyer's Appendix.
By Proposition \ref{polynomeg} we know that $L_{\alpha, \alpha'}$
is surjective. Hence 
$R \circ L_{\alpha, \alpha'}$ and $\Pi \circ L_{\alpha, \alpha'}$ are non-zero, and then define  hypersurfaces in ${\mathbb A}^{m+1}(\overline{\mathbb F}_q)$. Their
 numbers of rational points are  bounded respectively by $C_Rq^m$ and $C_{\Pi}q^m$
where $C_R$ and $C_{\Pi}$ are respectively  the degree of $R \circ L_{\alpha, \alpha'}$
and $\Pi \circ L_{\alpha, \alpha'}$
(see for example Section 5 of Chapter 1 in \cite{borevich1966number}).
Since $ L_{\alpha, \alpha'}$ is linear, one can bound $C_R$ and $C_{\Pi}$ by the degree $d_R$
of $R$ and the degree $d_{\Pi}$ of $\Pi$ and then one can bound $C_R+C_{\Pi}$ by
$\tilde{d}=d_R+d_{\Pi}$, 
which does not depend on the choice of $(\alpha,\alpha')$.
\end{proof}

\section{Geometric and arithmetic monodromy groups  of $D_{\alpha,\alpha'}^2f$}
\label{section:Geometric_and_arithmetic_mondormy_groups}

In the whole section we consider
a  polynomial $f$ of degree $m$ with $m\equiv 0, 1, 2$ or $7$ $\pmod 8$
and two distincts elements $\alpha, \alpha'$ of ${\mathbb F}_q^{\ast}$
such that the polynomial 
$g:=L_{\alpha,\alpha'}(f)$ is of degree exactly $d$ (given by Table \ref{table:definition_of_d}) and is Morse.
We denote by $u_0,\ldots,u_{d-1}$ the roots of $L_{\alpha,\alpha'}(f)(u)+t$, 
and for $i=0,\ldots,d-1$ we denote  
by $x_i$ a solution of the equation 
$$x(x+\alpha)(x+\alpha')(x+\alpha+\alpha')=u_i.$$
Hence
$D_{\alpha,\alpha'}^2f(x_i)=t$.
For convenience, we will note
$$S_{\gamma}(X)=X(X+\gamma)$$
for $\gamma \in \mathbb{F}_q$ and
$$
T_{\gamma_1,\gamma_2}(X)= 
X(X + \gamma_1)(X + \gamma_2) (X + \gamma_1+\gamma_2)
$$
for $(\gamma_1,\gamma_2)\in \mathbb{F}_q^2 $.
We will use the following equalities (easy to check):

\begin{equation}\label{equation:Sgamma}
S_{\gamma_1 \gamma_2}(x_i(x_i + \gamma_3))=u_i
 \textrm{ and } S_{\gamma_1 \gamma_2 \gamma_3 }
 (\gamma_3 x_i(x_i + \gamma_3))=\gamma_3^2 u_i
\end{equation}
where $\{ \gamma_1, \gamma_2, \gamma_3 \} = \{ \alpha , \alpha' , \alpha + \alpha' \}. $

 We
consider, for $i\in\{0,\ldots,d-1\}$, the  extensions $F(x_i)/F$ and $\Omega$ their compositum (where the field $F$ is defined in the previous section). 
Then $\Omega$ is the splitting field of $D_{\alpha,\alpha'}^2f(x)+t$ and
$\Gal(\Omega/{\mathbb F}_q(t))$ is the arithmetic monodromy group of $D_{\alpha,\alpha'}^2f$ 
whereas $\Gal(\Omega/{\mathbb F}_q^{\Omega}(t))$ is the geometric monodromy group
 of $D_{\alpha,\alpha'}^2f$, where we denote by ${\mathbb F}_q^{\Omega}$
the algebraic closure of ${\mathbb F}_q$ in $\Omega$.
The figure below sums up the situation whose details will be explained
in this section.

\begin{center}
\begin{tikzpicture}[node distance=2cm]
 \node (kt)             {${\mathbb F}_q(t)  $};

 \node (F)  [above of=kt]              {$F={\mathbb F}_q(u_0,\ldots,u_{d-1})$};
 \node (Fx0) [above of=F]  {$F(x_{0})$};
  \node (Fx0x1) [above of=Fx0]  {$F(x_{0},x_1)$};
  \node (factice) [above of=Fx0x1]  {};
  \node (Fpoint) [above of=factice]  {$\vdots$};
 \node (Omega)  [above of=Fpoint]   {$\Omega=F(x_0,\ldots,x_{d-1}) \ \ $};
\node (vide2) [right of=F] {}; 
 \node (komegat)   [right of=vide2]          {};
 \node (Fkomega)   [above of=komegat]          {$F{\mathbb F}_q^{\Omega}$};
 \node (Fkomegax0)   [above of=Fkomega]          {$F{\mathbb F}_q^{\Omega}(x_0)$};
  \node (Fkomegax0x1)   [above of=Fkomegax0]          {$F{\mathbb F}_q^{\Omega}(x_0,x_1)$};
  \node (facticebis)   [above of=Fkomegax0x1]          {$\vdots$};        
    \node (fin)   [above of=facticebis]          { $= \ F{\mathbb F}_q^{\Omega}(x_0,\ldots,x_{d-1})$};
 \draw (F)  to node[left, midway,scale=0.7]  {$\mathbb{Z} / 2 \mathbb{Z}\times \mathbb{Z} / 2 \mathbb{Z}$} (Fx0) ;
 \draw (Fx0)  to node[left, midway,scale=0.7]  {$\mathbb{Z} / 2 \mathbb{Z}\times \mathbb{Z} / 2 \mathbb{Z}$} (Fx0x1) ;
 \draw (Fx0)  -- (Fx0x1);
  \draw (Fpoint)  -- (Fx0x1);
  \draw (Fpoint)  -- (Omega);
  \draw (kt)  to node[left, midway,scale=0.85]  {$\Sgoth_{d}$} (F) ;
\draw (F) -- (Fkomega);
\draw (Fkomega) to node[right, midway,scale=0.7] {$\mathbb{Z} / 2 \mathbb{Z}\times \mathbb{Z} / 2 \mathbb{Z}$}  (Fkomegax0);
\draw (Fkomegax0x1) to node[right, midway,scale=0.7] {$\mathbb{Z} / 2 \mathbb{Z}\times \mathbb{Z} / 2 \mathbb{Z}$} (Fkomegax0);
\draw (Fkomegax0x1) -- (facticebis);
\draw (facticebis) -- (fin);
\draw (Fx0) -- (Fkomegax0);
\draw (Fkomegax0x1) -- (Fx0x1);

\end{tikzpicture}
\end{center}
The following lemma gives
conditions for two Artin-Schreier extensions to be equal.
\begin{lemma}\label{lemma:AS}
Let $k(y_1)$ and $k(y_2)$ be two Artin-Schreier extensions
of a field $k$
of characteristic $2$.
Suppose that $y_i^2+\gamma_iy_i=w_i$  
 for $i \in \{ 1,2 \}$
with $\gamma_i$ and $w_i$ in 
$k^{\ast}$.
Then  $k(y_1)=k(y_2)$  if and only if $\gamma_2 y_1 + \gamma_1 y_2 \in k$.
\end{lemma}
\begin{proof}
Suppose that $k(y_1)=k(y_2)$. Consequently $y_2 \in k(y_1)$
and there exists $(a,b) \in k^2$ such that
$y_2=a+by_1$.
Consider the element $\tau$ of $\Gal \left( k(y_1) / k \right)$ 
distinct from the identity. It maps $y_1$ to $y_1 + \gamma_1$. 
We have 
$\tau(y_2)=a+by_1+b \gamma_1$
 i.e. $\tau(y_2)=y_2 + b \gamma_1$.
But $\tau(y_2)$ is a root of 
$y^2+ \gamma_2 y = w_2$, 
so $\tau(y_2)=y_2$ or $\tau(y_2)= y_2+ \gamma_2$.
In the first case 
 $\tau$ would be the identity,
a contradiction.
Hence $\tau(y_2)= y_2+ \gamma_2$ and then $y_2+ \gamma_2 = y_2 + b \gamma_1$,
which implies that 
$\gamma_2=b\gamma_1$. So we get 
$\gamma_2 y_1 + \gamma_1 y_2=b\gamma_1y_1+\gamma_1y_2=b\gamma_1y_1+a\gamma_1+b\gamma_1y_1=a\gamma_1\in k$ where
we used that $y_2=a+by_1$.
The converse is straightforward. 
\end{proof}
Now we prove that a linear combination of the roots $u_j$
with no pole actually involves all of them.
\begin{lemma}\label{lemma_sum_u_i}
Let $\kappa$ be ${\mathbb F}_q$ or $\mathbb{F}_q^{\Omega}$.
For each place $\wp$ of $\kappa(u_0,\ldots,u_{d-1})$ above the place $\infty$
of $\kappa(t)$
and each $j \in \{ 0, \ldots, d-1 \}$ we have that $u_j$
has a simple pole at $\wp$.
Moreover, let $J \subset \{ 0, \ldots, d-1 \}$ and
let $c_0, \ldots, c_{d-1} \in \mathbb{F}_q^*$.
If $J$ is neither empty nor the whole set then
$ \sum_{j \in J} c_j u_j $  has a pole at a place
of $\kappa(u_0,\ldots,u_{d-1})$ lying
over the infinite place
$\infty$ of $\kappa(t)$.
\end{lemma}
\begin{proof}
Fix $\wp$ a place above $\infty$
and $u_i$ a root of $g(u)-t$.
We have
$v_{\wp}(g(u_i))=v_{\wp}(t)$
and 
$v_{\wp}(t)=e \left( \wp \vert \infty \right)  v_{\infty} (t)$
where $e \left( \wp \vert \infty \right)$ is the ramification index of $\wp$ over $\infty$.
By \cite{Serre}, p. 41, we have that the inertia group at infinity is generated by a $d$-cycle, so we have  $e \left( \wp \vert \infty \right)=d$ and then
$v_{\wp}(t)=-d$.
Now $v_{\wp}(g(u_i))=v_{\wp} \left( b_0 u_i^d+b_1u_i^{d-1}+\cdots+b_d \right)$
so using the properties of the valuation of a sum 
we deduce that 
$v_{\wp}(u_i)=-1$.

The proof of the second part of the lemma is inspired by \cite{Voloch}.
To obtain a contradiction, suppose that 
$J \subset \{ 0, \ldots, d-1 \}$ and that
$j_0 \in J$ whereas $j_1 \in \{ 0, \ldots, d-1\} \setminus J$.
Suppose also that 
$ \sum_{j \in J} c_j u_j $ has no pole in places above $\infty$.
Then it has no pole at all, and so it is constant, i.e.
it belongs to $ \kappa$.
By Proposition \ref{Monodromy} we have that
$\Gal \left( \kappa(u_0,\ldots,u_{d-1}) / 
\kappa(t) \right)$
is  $\Sgoth_d$. Let us choose the automorphism $\theta$ corresponding to the transposition  $(j_0 j_1)$ and let us apply $\theta$ to $\sum_{j \in J} c_j u_j$. We  obtain
$ \sum_{j \in J \setminus j_0} c_{j} u_j + c_{j_0} u_{j_0} =  
\sum_{j \in J \setminus j_0} c_{j} u_j + c_{j_0} u_{j_1}$.
We deduce $u_{j_0}=u_{j_1}$, 
a contradiction.
\end{proof}

The following lemma, used with Lemma \ref{lemma:AS},
 will enable us to distinguish different Artin-Schreier subextensions
 of $\Omega$.
\begin{lemma}\label{lemma_sum_x_i}
Let $\widetilde{F}$ be $F$ or $F\mathbb{F}_q^{\Omega}$.
Let $J$ be a non-empty strict subset of $\{0, \ldots, d-1 \}$
and for all $j \in J$ 
consider any 
$\gamma_j \in \{ \alpha, \alpha', \alpha + \alpha' \}$.
 Then
$$ \sum_{j \in J} \gamma_j x_j(x_j+\gamma_j) \notin \widetilde{F}.$$
\end{lemma}
\begin{proof}
In order to obtain a contradiction
suppose that $\sum_{j \in J}\gamma_j  x_j (x_j+\gamma_j) \in \widetilde{F}$.
Lemma \ref{lemma_sum_u_i} implies that 
 $ \sum_{j \in J} \gamma_j^2u_j $  has a pole 
 at a place $\wp$  of $\widetilde{F}$ above $\infty$.
Moreover this pole is simple as for all $j\in \{ 1, \ldots, d-1 \} $ 
the root $u_j$ has a simple pole by Lemma \ref{lemma_sum_u_i}.
Now consider 
$A = \sum_{j \in J} \gamma_j x_j(x_j+\gamma_j)$ and
 $B= \sum_{j \in J} \gamma_j x_j(x_j+\gamma_j)  + \alpha \alpha' (\alpha+\alpha')$.
If $A$ (and thus $B$) belongs to $\widetilde{F}$,
one can consider the valuation of $A$ and $B$ at $\wp$.
As 
\[ A.B = 
S_{\alpha \alpha' (\alpha+\alpha')} \left(
\sum_{j \in J} \gamma_j x_j(x_j+\gamma_j) \right)  = \sum_{j \in J} \gamma_j^2 u_j, \]
it follows that
either $A$ or $B$ has a pole.
Since $A$ and $B$ differ by a constant,
it follows that both of them have a pole and  
the order of multiplicity 
is the same.
Thus we obtain $2 v_{\wp}(A)=-1$ which is a contradiction.
\end{proof}

The following lemma establishes the base case of the induction proof of Proposition
\ref{proposition:Galois_second_step}.

\begin{lemma}\label{lemma:case_n_eq_0}
Let $\widetilde{F}$ be $F$ or $F\mathbb{F}_q^{\Omega}$.
Let $i \in \{ 0, \ldots, d-1 \}$.
The field $\widetilde{F}(x_i)$ is a degree $4$ extension of $\widetilde F$
and its Galois group is 
${\mathbb Z}/2{\mathbb Z}\times {\mathbb Z}/2{\mathbb Z}$.
The three subextensions of degree $2$ are the subextensions
$\widetilde{F} \left( x_i(x_i+\gamma) \right) $
 where $\gamma \in \{ \alpha, \alpha', \alpha+\alpha' \}$. 
The following diagram sums up the situation:
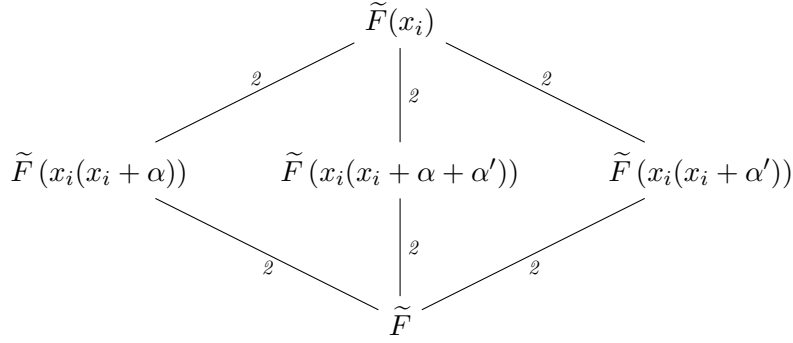
\begin{center}
\begin{tikzpicture}[node distance=2cm]
 \node (k)                  {$\widetilde{F}$};
 \node (k1plus2) [above of=k]  {$\widetilde{F} \left( x_i ( x_i+ \alpha + \alpha' ) \right) $};
 \node (k12)  [above of=k1plus2]   {$ \widetilde{F}(x_i)$};
 \node (k2)  [right of=k1plus2]  {};
 \node (k2bis)  [right of=k2]  {$\widetilde{F} \left( x_i ( x_i +  \alpha') \right) $};
 \node (k1)  [left of=k1plus2]   {};
 \node (k1bis)  [left of=k1]   {$\widetilde{F} \left( x_i ( x_i+ \alpha ) \right) $};

 \draw (k)   -- (k2bis) node [midway,right,below,scale=0.7] {2} ;
 \draw (k)   -- (k1plus2) node [midway,right,scale=0.7] {2};
 \draw (k)   -- (k1bis) node [midway,left,below,scale=0.7] {2} ;
 \draw (k2bis)  -- (k12) node [midway,above,scale=0.7] {2} ;
 \draw (k1bis)  -- (k12) node [midway,left,above,scale=0.7] {2} ;
 \draw (k1plus2)  -- (k12) node [midway,right,scale=0.7] {2} ;
\end{tikzpicture}
\end{center}
\end{lemma}
\begin{proof}
First notice that
$x_i \notin \widetilde{F}$.
Otherwise, one would obtain a contradiction
 considering the equality
 $x_i(x_i + \alpha)(x_i + \alpha')(x_i + \alpha + \alpha')=u_i$,
the valuation of $x_i$ at a place above $\infty$,
and the valuation of $u_i$ at this place which is $-1$.
Now suppose that $[\widetilde{F}(x_i):\widetilde{F}]=2$.
We would have a degree $2$ factor of the 
polynomial 
$X(X +\alpha)(X + \alpha')(X + \alpha + \alpha') + u_i$
and then an element $x_i(x_i + \gamma)$ 
with $\gamma \in \{ \alpha , \alpha', \alpha+\alpha'\}$
would be in $\widetilde{F}$, contradicting Lemma \ref{lemma_sum_x_i}.
So $[\widetilde{F}(x_i):\widetilde{F}]=4$, and $T_{\alpha, \alpha'}(X)+u_i$
is the minimal polynomial of $x_i$ over $\widetilde{F}$.
It enables us to define, 
for any $\gamma \in \{ \alpha, \alpha', \alpha + \alpha'\}$,
an element $\tau_{\gamma}$ of $\Gal \left( \widetilde{F}(x_i) / \widetilde{F} \right)$
by $\tau_{\gamma} (x_i) = x_i + \gamma$.
We thus have $\Gal \left( \widetilde{F}(x_i) / \widetilde{F} \right) = 
\{ id, \tau_{\alpha} , \tau_{\alpha'} , \tau_{\alpha + \alpha'}  \} $
and thus $\Gal \left( \widetilde{F}(x_i) / \widetilde{F} \right) \simeq {\mathbb Z}/2{\mathbb Z}\times {\mathbb Z}/2{\mathbb Z}$.
There are three subextensions of degree $2$, namely
the subextensions $\widetilde{F} \left( x_i(x_i + \gamma ) \right) $
where
$\gamma \in \{  \alpha, \alpha', \alpha + \alpha'\}$.
Their stabilizers are respectively the index 2 subgroups $\{ \id, \tau_{\gamma} \}$.
\end{proof}
The previous lemmas enable us to determine in the following two propositions the Galois groups of 
$\widetilde{F}(x_0, \ldots,x_{d-2})$ and 
$\Omega = \widetilde{F}(x_0, \ldots,x_{d-1})$
over $\widetilde{F}$ 
 where $\widetilde{F}$
is equal to $F$ or $F \mathbb{F}_q^{\Omega}$.

\begin{proposition}\label{proposition:Galois_second_step}
Let ${\tilde F}$ be $F$ 
or $F \mathbb{F}_q^{\Omega}$ and
let $r$ be an integer such that $0 \leqslant r\leqslant d-2$. Then: \\ \ \\
(i) The field $\widetilde{F}(x_0, \ldots,x_r)$ is an extension of degree $4^{r+1}$ of $\widetilde{F}$.\\ \ \\
(ii) The Galois group $\Gal \left( \widetilde{F}(x_0, \ldots,x_r) / \widetilde{F} \right)$ 
 is $\left( {\mathbb Z} / 2{\mathbb Z} \times {\mathbb Z} / 2 {\mathbb Z}\right)^{r+1}$.
 It is generated by the automorphisms $\tau_{i,\gamma}$
 for $i\in \{ 0, \ldots, r \}$ and $\gamma \in \{ \alpha, \alpha', \alpha + \alpha'\}$
 (where $\tau_{i,\gamma}$
 maps $x_i$ to $x_i + \gamma$ and leaves $x_j$ invariant
  for $j \neq i$). \\ \ \\
(iii) There are $4^{r+1}-1$ quadratic extensions  of
$\widetilde{F}$ contained in $\widetilde{F}(x_0, \ldots,x_r)$. 
These extensions are 
the fields $\widetilde{F} \left( \sum_{j \in J} \gamma_j x_j(x_j+\gamma_j) \right) $
with non-empty $J \subset \{ 0, \ldots, r \}$ and 
$\gamma_j \in \{ \alpha, \alpha', \alpha + \alpha'\}$ for all 
$j \in J$.
\end{proposition}
\begin{proof}
We proceed by induction.
The case $r=0$ follows from 
Lemma \ref{lemma:case_n_eq_0}.
Assuming that the proposition holds for $r-1$, with $0<r \leqslant d-2$,
we will prove it for $r$.
We consider the extension
$\widetilde{F}(x_0, \ldots,x_{r-1})(x_r)$ of $\widetilde{F}(x_0, \ldots,x_{r-1})$.
We first prove that the degree of this extension is $4$ and that
the minimal polynomial of $x_r$ is $T_{\alpha, \alpha'}(X)+u_r$.
Suppose it is false: either $x_r \in \widetilde{F}(x_0, \ldots,x_{r-1})$ 
or $T_{\alpha,\alpha'}(X)+u_r$ (which is equal to $(x+x_r)(x+x_r+\alpha)(x+x_r+\alpha')(x+x_r+\alpha+\alpha')$) has a degree 2 factor
in $\widetilde{F}(x_0, \ldots,x_{r-1})[X]$,
hence there exists
 $\gamma \in \{ \alpha, \alpha', \alpha + \alpha' \} $
such that
$x_r (x_r + \gamma ) \in \widetilde{F}(x_0, \ldots,x_{r-1})$.
In both cases we would have an extension 
$\widetilde{F} \left(  x_r(x_r+\gamma) \right) $ of degree $2$ of
$\widetilde{F}$ contained in $\widetilde{F}(x_0, \ldots,x_{r-1})$.
Use the induction hypothesis: it is one of the subextensions
$\widetilde{F} \left( \sum_{j \in J} \gamma_j x_j(x_j+\gamma_j) \right) $
with  a non-empty subset $J \subset \{ 0, \ldots, r-1 \}$.
By Lemma \ref{lemma:AS} and identities (\ref{equation:Sgamma})
it follows that 
$ \sum_{j \in J} \gamma_j x_j(x_j+\gamma_j) + \gamma x_r(x_r+\gamma) \in \widetilde{F}$,
a contradiction with Lemma \ref{lemma_sum_x_i}. We conclude that the extension $ \widetilde{F}(x_0, \ldots,x_r) /  \widetilde{F}(x_0, \ldots,x_{r-1})$ has degree $4$ and then
$\widetilde{F}(x_0, \ldots,x_r) /  \widetilde{F}$ has degree $4^{r+1}$.

But we can  define $4^{r+1}$ different $\widetilde{F}$-automorphisms of $\widetilde{F}(x_0, \ldots,x_r)$ by sending for any 
$i\in \{ 0, \ldots, r \}$ the element $x_i$ to $x_i+\gamma_i$ with
 $\gamma_i \in \{0, \alpha, \alpha', \alpha + \alpha'\}$.
Since all these automorphisms (apart from the identity) have order 2, the Galois group $\Gal \left( \widetilde{F}(x_0, \ldots,x_r) / \widetilde{F} \right)$
is isomorphic to $\left( {\mathbb Z} / 2{\mathbb Z} \times {\mathbb Z} / 2 {\mathbb Z}\right)^{r+1}$.

 For any non-empty subset $J \subset \{ 0, \ldots, r\}$
 and for any choice of a family $(\gamma_j)_{ j \in J}$ of elements
  of $\{ \alpha, \alpha', \alpha + \alpha' \}$,
 we know that $\sum_{j \in J} \gamma_j x_j(x_j+\gamma_j)$ is a root
 of $S_{\alpha \alpha' (\alpha + \alpha')} (X) + \sum_{j \in J} \gamma_j^2 u_j$.
By Lemma  \ref{lemma_sum_x_i} we also know that this sum does not belong to $\widetilde{F}$,
 so the extensions 
 $\widetilde{F} \left( \sum_{j \in J} \gamma_j x_j(x_j+\gamma_j) \right)$
are quadratic.

We claim that we obtain this way $4^{r+1}-1$ different
quadratic extensions between $\widetilde{F}$ and 
 $\widetilde{F}(x_0, \ldots,x_r)$.
To prove our claim, we consider two families
$(\gamma_j)_{j \in J}$ and 
$(\gamma'_j)_{j \in J'}$ of elements of $\{ \alpha, \alpha' , \alpha+ \alpha'\} $
where $J$ and $J'$ are two subsets of $\{ 0 ,  \ldots ,  r \}$.
We notice that if $j \in J \cap J'$ is such that $\gamma_j \neq \gamma'_j$ then
$\gamma_j x_j(x_j+\gamma_j) + \gamma'_j x_j(x_j+\gamma'_j) = 
\gamma''_j x_j(x_j+\gamma''_j) $ where 
$\{ \gamma_j, \gamma'_j, \gamma''_j \} = \{ \alpha, \alpha', \alpha + \alpha' \}$.
Then if
$\widetilde{F} \left( \sum_{j \in J} \gamma_j x_j(x_j+\gamma_j) \right)
 = \widetilde{F} \left( \sum_{j \in J'} \gamma_j x_j(x_j+\gamma_j) \right)$
we obtain by  Lemma \ref{lemma:AS}
a sum
$$ 
\sum_{j \in J \setminus J'} \gamma_j x_j(x_j+\gamma_j) + 
\sum_{j \in J' \setminus J } \gamma'_j x_j(x_j+\gamma'_j) +
\sum_{j \in J \cap J' \atop \gamma_j \neq \gamma'_j} \gamma_j'' x_j(x_j+\gamma_j'') 
$$
which is in $\widetilde{F}$.
By Lemma \ref{lemma_sum_x_i}, it implies $J=J'$ and $\gamma_j = \gamma_j'$ for all $j \in J$.

Finally, we claim that these $4^{r+1}-1$ quadratic extensions are the only ones.
 Indeed, the quadratic extensions are in correspondence 
 with the subgroups of $\left( {\mathbb Z} / 2 {\mathbb Z}\times {\mathbb Z} / 2 {\mathbb Z}\right)^{r+1}$
 of index $2$.
These subgroups are 
the hyperplanes of  $\left( {\mathbb Z} / 2{\mathbb Z} \right)^{2r+2}$
and there are $4^{r+1}-1$ such hyperplanes.
\end{proof}
\bigskip
Recall that in this section  the polynomial $g=L_{\alpha,\alpha'}(f)=\sum_{i=0}^db_{d-i}x^i$ 
is supposed to be Morse and to have degree exactly $d$. 
We can now establish the main result of this section:
we give a sufficient condition on $b_1/b_0$
for $\Omega/{\mathbb F}_q(t)$ to be regular, 
which is a necessary condition 
to apply the Chebotarev theorem.
\begin{proposition}\label{regular}
If there exists $x\in{\mathbb F}_q$ such that 
$$\frac{b_1}{b_0}=x(x+\alpha)(x+\alpha')(x+\alpha+\alpha')$$
 then we have: \\ \ \\
(i) $ F(x_0,\ldots,x_{d-2},x_{d-1})=F(x_0,\ldots,x_{d-2})$. \\ \ \\
(ii)  $  \Gal(\Omega/F) 
 \simeq \Gal(\Omega/F{\mathbb F}_q^{\Omega}) 
 \simeq
  \bigl({\mathbb Z}/2{\mathbb Z}\times {\mathbb Z}/2{\mathbb Z}\bigr)^{d-1} $.\\ \ \\
(iii) The Galois group $\Gal(\Omega/{\mathbb F}_q(t))$
is an extension of  $\Sgoth_d$ by  
 $\bigl({\mathbb Z}/2{\mathbb Z}\times {\mathbb Z}/2{\mathbb Z}\bigr)^{d-1}$.\\ \ \\
(iv) $\Omega/{\mathbb F}_q(t)$ is a regular extension i.e. ${\mathbb F}_q^{\Omega}={\mathbb F}_q$.
\end{proposition}
\begin{proof}
Suppose that there exists $x \in \mathbb{F}_q$
 such that $b_1/b_0=T_{\alpha,\alpha'}( x)$.
We have
$\frac{b_1}{b_0}=\sum_{i=0}^{d-1}u_i=\sum_{i=0}^{d-1} T_{\alpha,\alpha'}(x_i)$ and then
by linearity we deduce that $T_{\alpha,\alpha'} ( x_{d-1} + x
 + \sum_{i=0}^{d-2} x_i ) =0$.
It implies that $x_{d-1} + x + \sum_{i=0}^{d-2} x_i  \in \{0, \alpha , \alpha' , \alpha + \alpha'  \}$
and thus $x_{d-1} \in F (x_0, \ldots , x_{d- 2})$ which proves the point {\it (i)}. 
Using {\it (i)} and Proposition \ref{proposition:Galois_second_step} we obtain the point {\it (ii)}.
Now point {\it (ii)} with Proposition \ref{Monodromy} and Galois theory give point {\it (iii)}.
To obtain point {\it (iv)}, we use the multiplicativity of the degrees in fields extensions and
we write
$[\Omega:F]=[\Omega : F {\mathbb F}_q^{\Omega}]\times
[F{\mathbb F}_q^{\Omega}:F]$. 
Points {\it (i)} and {\it (ii)} yield $[F{\mathbb F}_q^{\Omega}:F]=1$ and then the extension
$\Omega/F$ is regular.
But Proposition \ref{Monodromy} implies that the extension $F/{\mathbb F}_q(t)$ is regular.
Then we obtain that the extension 
$\Omega/{\mathbb F}_q(t)$ is regular.
\end{proof}


\section{Application of Chebotarev density theorem}\label{section_chebotarev}

The Chebotarev density theorem describes the proportion of places
splitting in a given way 
in Galois extensions of global fields
(see \cite{Rosen} p. 125).
In \cite{FouqueTibouchi}, P. Fouque and M. Tibouchi made the following version
of Chebotarev theorem explicit. They deduced it from 
the Proposition 4.6.8 in \cite{FriedJarden}. 
\begin{theorem} (\textbf{Chebotarev})\label{Chebotarev}
Let $K$ be an extension of ${\mathbb F}_q(t)$ of finite degree $d_K $ and $L$ a Galois extension of $K$ 
of finite degree $d_{L/K}$.
 Assume ${\mathbb F}_q$ is algebraically closed in $L$, and fix some subset $S$ of $\Gal(L/K)$ stable under conjugation. 
 Let $s =\sharp S$ and let $N(S)$ be the number of places $v$ of $K$
of degree 1, unramified in $L$, such that the Artin symbol $\bigl(\frac{L/K}{v}\bigr)$ (defined up to conjugation) is in $S$ . Then
$$\left|  N(S)-\frac{s}{d_{L/K}}q \right|  \leqslant \frac{2s}{d_{L/K}}\bigl((d_{L/K}+g_L)q^{1/2}+d_{L/K}(2g_K+1)q^{1/4}+g_L+d_Kd_{L/K}\bigr)$$
where $g_K$ and $g_L$ are the genera of the function fields $K$ and $L$.
\end{theorem}
In this work, we are interested in places of $K=\mathbb{F}_q(t)$
which split completely in $L=\Omega$. 
Indeed, if a place of degree one $(t- \beta)$ with $\beta \in \mathbb{F}_q$ totally splits
in $\Omega$, then the polynomial $D^2_{\alpha,\alpha'} f (x) - \beta$ totally splits in
$\mathbb{F}_q[x]$.
These places correspond to places $v$ of $K$ which are unramified in $\Omega$ and for which
the Artin symbol
 $\bigl(\frac{\Omega/\mathbb{F}_q(t)}{v}\bigr)$ is equal to $(\id)$, 
 the conjugacy class of $\Gal(\Omega/\mathbb{F}_q(t))$ consisting of the identity element.
Hence the previous theorem can be used to prove 
the following proposition 
which will be the main tool
to prove Theorem \ref{maintheorem}.
\begin{proposition}\label{application_Chebotarev}
Let $m \geqslant 7 $ be an integer and $d$ 
as defined in Definition \ref{definition_de_d}.
There exists an integer $N$ depending only on $d$
such that for all $n \geqslant N$,
for all $f \in \mathbb{F}_q[x]$ (with $q=2^n$) of degree less or equal to $m$,
and for all couple $(\alpha,\alpha')$ of disctinct elements of $\mathbb{F}_q^*$
such that
the extension $\Omega/{\mathbb F}_q(t)$ is regular
there exists $\beta\in {\mathbb F}_q$ such that the polynomial 
$D^2_{\alpha,\alpha'}f(x)+\beta$ 
splits in $ \mathbb{F}_q[x]$ with no repeated factors. 
\end{proposition}
\begin{proof}
Since the extension $\Omega/\mathbb{F}_q(t)$ is regular,
by the above Chebotarev theorem 
the number $N(S)$ of 
places $v$ of ${\mathbb F}_q(t)$
of degree 1, unramified in $\Omega$, 
such that $\bigl(\frac{\Omega/{\mathbb F}_q(t)}{v}\bigr)=(\id)$
satisfies
$$N(S)\geqslant \frac{q}{d_{L/K}}-2 \bigl((1+\frac{g_L}{d_{L/K}})q^{1/2}+q^{1/4}+1+\frac{g_L}{d_{L/K}}\bigr).$$
From the point $(iii)$ of
Proposition \ref{regular}
 we know that
 $d_{L/K} = d! 4^{d-1}$ or  $d_{L/K} = d! 4^d$.
Moreover, one can obtain an upper bound on $g_L$ depending only on $d$
using induction and Castelnuovo's inequality
as stated in Theorem 3.11.3 of \cite{stichtenoth2009algebraic}.
Then if $q$ (or $n$ since $q=2^n$) is sufficiently large, we will have $N(S)\geqslant 1$, 
which concludes the proof.
\end{proof}
\section{A class of good polynomials}\label{section_good}
The last proposition applies when the Galois closure of 
$D^2_{\alpha,\alpha'}f-t$ is regular.
By Proposition \ref{regular} this is the case when the quotient
of the first coefficients 
of $L_{\alpha,\alpha'}(f)$ can be written in the form 
$x(x+\alpha)(x+\alpha')(x+\alpha+\alpha')$ with $x \in \mathbb{F}_q$.
Our strategy is now to choose a well fitted
finite family $(\alpha_i, \alpha'_i)_{i \in \{ 1, \ldots, k \} }$ 
such that we can apply Proposition \ref{application_Chebotarev}
with at least one couple $(\alpha_i, \alpha'_i)$ for most of polynomials of degree
$m$.

\begin{proposition}\label{covers}
Let $\varepsilon >0$.
There exist $k \in \mathbb{N}^*$ and $N \in \mathbb{N}^*$
such that for all $n \geqslant N$
there exist $k$ couples
$(\alpha_1,\alpha_1'), \ldots ,(\alpha_k,\alpha_k')$
of distinct elements of $\mathbb{F}_q^*$
such that there exist at least 
$(1-\varepsilon)q^m(q-1)-q^m$ polynomials $f\in{\mathbb F}_q[x]$ of degree $m$ 
such that:  \\
- for all $i \in \{ 1, \ldots,k \} $ 
the polynomial $L_{\alpha_i,\alpha_i'}(f)=b_0 x^d + b_1 x^{d-1} + \cdots +b_d$
has degree $d$ and \\
- for at least one of the couples $(\alpha_i, \alpha_i')$,
the equation
$$\frac{b_1}{b_0} =x(x+\alpha_i)(x+\alpha_i')(x+\alpha_i+\alpha_i')$$
has a solution in ${\mathbb F}_q$.
\end{proposition}
\begin{proof}
Let $f=\sum_{j=0}^m a_j x^{m-j}$ be
a polynomial of degree $m$.
First we notice that for any
 distinct elements $\alpha$ and $\alpha'$ of ${\mathbb F}_q^{\ast}$
the polynomial $L_{\alpha,\alpha'}(f)$ is of degree $d$ 
(with $d$ given by Table \ref{table:definition_of_d})
if and only if 
$a_{j_1} \neq 0$, where $j_1 \in \{ 0, 1, 2, 3 \}$ is given by Lemma \ref{b1overb0}.
In this case, the quotient
$b_1/b_0$ is well defined.
By abuse of notation,
we will write
$\frac{b_1}{b_0} \left( L_{\alpha,\alpha'} (f) \right) $
for this quotient. 
By linearity 
of $L_{\alpha,\alpha'}$ 
we have 
$b_1/b_0 \left( L_{\alpha,\alpha'} (\lambda f )\right) = 
b_1/b_0 \left( L_{\alpha,\alpha'} (f )\right)$
for any $\lambda \in \mathbb{F}_{q}^*$.
So in order to count the polynomials $f$ satisfying
the conditions of the proposition we can restrict ourselves to those whose coefficient $a_{j_1}$ 
is $1$, and then multiply by $q-1$ in our count. 
 We will denote by 
$\mathcal{P}_{j_1}$ the
set of polynomials $f\in {\mathbb F}_q[x]$
of degree $m$ such that $a_{j_1} = 1$ and we will
identify $\mathcal{P}_{j_1}$ with
${\mathbb F}_q^{m}$.

Let $\varepsilon >0$. 
Consider $k$ such that $(3/4)^k< \varepsilon$, and $N=2k$. 
For $n \geqslant N$, 
identify $\mathbb{F}_{2^n}$ with $\mathbb{F}_{2}^n$ and
fix a basis.
Consider $k$ couples 
$(\alpha_1,\alpha_1'),  \ldots, (\alpha_k,\alpha_k')$ 
of distinct elements of $\mathbb{F}_q^*$ such that for any $i \in \{ 1, \ldots, k \}$
the subspace $ \Ima T_{\alpha_i,\alpha_i'}$ has for equation $(\xi_{2i-1}=\xi_{2i}=0)$
in the  fixed basis of
$\mathbb{F}_{2}^{n}$ 
(recall that $T_{\alpha,\alpha'}$ is defined in Section 
\ref{section:Geometric_and_arithmetic_mondormy_groups} by
$T_{\alpha,\alpha'} (x) = x(x+\alpha)(x+\alpha')(x+\alpha+\alpha')$).
The existence of these couples is given by
Lemma \ref{lemma:representation}.
For any $i \in \{ 1, \ldots, k \}$
we consider the map
 $\psi_i \ : \ \mathcal{P}_{j_1} \rightarrow \mathbb{F}_{q}$ 
defined by $\psi_i (f) = b_1/b_0 \left( L_{\alpha_i,\alpha_i'} (f) \right)$.
Lemma \ref{b1overb0} gives the existence of 
an integer $j_2$
(which depends only on the congruence of $m$)
and the existence of coefficients $c_{i,j}$ and $d_i$ 
in $\mathbb{F}_{q}$ such that
$$\psi_i (f) = 
 a_{j_2} + d_i + 
\sum_{\scriptstyle j\in \{ 0, \ldots, m \} \setminus  \{  j_1, j_2\} } 
c_{i,j}
a_j
.
$$
Now, for $i \in \{ 1, \ldots, k \}$ the set of 
 $(a_0, \ldots, a_{{j_1}-1},a_{{j_1}+1}, \ldots, a_m) \in 
\mathbb{F}_{2^n}^{m}  $ corresponding to elements of
$\psi_i^{-1} \left( \Ima T_{\alpha_i,\alpha_i'} \right)$
is an affine space over $\mathbb{F}_{2}$ which is the intersection
of the affine hyperplanes given by the affine equations
$(a_{j_2})_{2i-1} + \sum_{j \notin \{  j_1,j_2 \} } (c_{i,j} a_j)_{2i-1}     = (d_i)_{2i-1}$ and
$(a_{j_2})_{2i} +  \sum_{j \notin \{  j_1,j_2 \} } (c_{i,j} a_j)_{2i}    = (d_i)_{2i}$.
The $2k$ linear forms defined by the left-hand sides of these equations are linearly independant,
so a change of basis of the $\mathbb{F}_2$-vector space $\mathbb{F}_{2}^{nm}$ gives the following
systems of equations of $\psi_i^{-1} \left( \Ima T_{\alpha_i,\alpha_i'} \right)$:
$\zeta_{2i-1}=\mu_i$ and $\zeta_{2i}=\nu_i$ where 
$(\mu_i)_{i \in \{ 1, \ldots ,k \} }$ and  $(\nu_i)_{i \in \{ 1 ,\ldots, k \} }$ are elements of
$\mathbb{F}_2^{k}$.
To count the elements $\zeta \in \mathbb{F}_{2}^{nm}$
such that $\zeta $
corresponds to an element of $\cup_{i=1}^k \psi_i^{-1} \left( \Ima T_{\alpha_i,\alpha_i'} \right)$
one can determine the cardinal of the complementary.
For each $i \in \{ 1, \ldots, k \}$
there are three ways to choose the couple of components $(\zeta_{2i-1},\zeta_{2i})$ 
different from $(\mu_i,\nu_i)$, and $2^{mn-2k}$ ways to choose the other components.

We find $\# \cup_{i=1}^k \psi_i^{-1} \left( \Ima T_{\alpha_i,\alpha_i'} \right) 
=   2^{mn}-3^k 2^{mn-2k}  =  q^m \left( 1-(3/4)^k \right)$.
Finally, we have to multiply by $q-1$ in order to 
take into account the coefficient $a_{j_1}$, and to
remove the $q^m$ polynomials of degree less than $m$.
(Note that in the case where 
$m \equiv 7 \mod (8)$ we have already removed these polynomials
as we have supposed $a_{j_1} \neq 0$
and in this case $j_1=0$.)

\end{proof}

\begin{lemma}\label{b1overb0}
Let $f=\sum_{j=0}^m a_j x^{m-j}$ be a polynomial of $\mathbb{F}_q[x]$ of degree $m$
with $m \equiv 0, 1, 2$ or $7 \pmod 8$.
For $\alpha,\alpha' \in \mathbb{F}^{\ast}_q$ we set
$L_{\alpha,\alpha'} (f) = \sum_{j=0}^d b_j x^{d-j}$.
We have $b_0=\alpha \alpha' (\alpha + \alpha') a_i$ where $i \in \{0, 1, 2, 3 \}$ satisfies
$i \equiv m+1 \mod 4$.
Moreover the following table gives the quotient $b_1/b_0$ 
as a function of the coefficients of $f$
 depending on the congruence of $m$ modulo $16$.
\begin{center}
\begin{tabular}{|c|c|} 
\hline
m & $b_1/b_0$  \\
(16)   & \\
\hline
0 & $\left(  
(\alpha^2 \alpha' + \alpha'^2 \alpha) a_2
+ (\alpha^2 +\alpha \alpha' + \alpha'^2) a_3
+ a_5
 \right) a_1^{-1}$ \\
1 & $\left(  
(\alpha^2 \alpha' + \alpha'^2 \alpha) a_3
+ (\alpha^2 +\alpha \alpha' + \alpha'^2) a_4
+ a_6
 \right) a_2^{-1}$ \\
2 & $\left(  
(\alpha^2 \alpha' + \alpha'^2 \alpha) a_4
+ (\alpha^2 +\alpha \alpha' + \alpha'^2) a_5
+ a_7
 \right) a_3^{-1}$\\
7 &  $ \left( (\alpha^2 \alpha' + \alpha'^2 \alpha)a_1 +(\alpha^2 +\alpha \alpha' + \alpha'^2) a_2 +a_4 \right) a_0^{-1} 
+\alpha^4 +  \alpha^2 \alpha'^2+ \alpha'^4  $ \\
8 & $\left(  
(\alpha^2 \alpha' + \alpha'^2 \alpha) a_2
+ (\alpha^2 +\alpha \alpha' + \alpha'^2) a_3
+ a_5
 \right) a_1^{-1}
+ \alpha^4 + \alpha^2 \alpha'^2  + \alpha'^4 $ \\
9 & 
{ \scriptsize $\left(  
\sum_{i=0}^6 \alpha^i \alpha'^{6-i} a_0
+ (\alpha^2 \alpha' + \alpha'^2 \alpha) a_3
+ (\alpha^2 +\alpha \alpha' + \alpha'^2) a_4
+ a_6
 \right) a_2^{-1}
+ \alpha^4 + \alpha^2 \alpha'^2  + \alpha'^4 $ 
} 
\\
10 &  
{ \scriptsize
$\left(  
a_0 \sum_{i=1}^6 \alpha^i \alpha'^{7-i} +
a_1\sum_{i=0}^6 \alpha^i \alpha'^{6-i} 
+ (\alpha^2 \alpha' + \alpha'^2 \alpha) a_4
+ (\alpha^2 +\alpha \alpha' + \alpha'^2) a_5
+ a_7
 \right) a_3^{-1} $}\\
 & { \scriptsize $+ \alpha^4 + \alpha^2 \alpha'^2  + \alpha'^4 $ } \\
15 & $ \left( (\alpha^2 \alpha' + \alpha'^2 \alpha)a_1 +(\alpha^2 +\alpha \alpha' + \alpha'^2) a_2 +a_4 \right) a_0^{-1}$ \\ 
\hline
\end{tabular}
\end{center}
\end{lemma}
\begin{proof}
The question amounts to solving the linear system
\begin{equation}\label{equation:equationL}
\sum_{j=0}^d b_j 
\left( x(x+\alpha)(x+\alpha')(x+\alpha+\alpha') \right)^{d-j} =
D^2_{\alpha,\alpha'}\left( \sum_{j=0}^m a_j x^{m-j} \right).
\end{equation}
On the one hand we have
$$
D^2_{\alpha,\alpha'} f(x) = 
\sum_{j=1}^m \left( \sum_{s=1}^j a_{j-s} C_s \binom{m-j+s}{s}\right) x^{m-j}
$$
where $C_s$ denotes $\alpha^s + \alpha'^s+(\alpha + \alpha')^s$  for $s \geqslant 1$ .
We notice that $C_1=C_2=C_4=0$ and that $C_3=\alpha \alpha' (\alpha + \alpha')$.
It implies
\begin{multline*}
D^2_{\alpha,\alpha'} f(x) = 
\binom{m}{3} a_0 C_3 x^{m-3} 
+ \binom{m-1}{3} a_1 C_3 x^{m-4} \\ 
+   \left( \binom{m}{5} a_0 C_5 + \binom{m-2}{3} a_2 C_3 \right) x^{m-5} \\ 
+ \left( \binom{m}{6} a_0 C_6 + \binom{m-1}{5} a_1 C_5 +
\binom{m-3}{3} a_3 C_3 \right) x^{m-6} + \cdots 
\end{multline*}  
On the other hand, the left-hand side of (\ref{equation:equationL})
is equal to
\begin{multline*}
g \left( T_{\alpha,\alpha'}(x) \right) =  b_0x^{4d} + b_0 d (\alpha^2 + \alpha'^2 + \alpha\alpha' ) x^{4d-2} +
 b_0 d (\alpha + \alpha')  \alpha \alpha' x^{4d-3}  \\
 +  \left( b_0 \binom{d}{2} (\alpha^2 + \alpha'^2 + \alpha\alpha' )^2 + b_1 \right) x^{4d-4}
+ \cdots
\end{multline*}  
To obtain $b_0$ (and respectively $b_1$) 
one can identify the coefficients of $x^{4d}$ (respectively $x^{4d-4}$)
on both sides of (\ref{equation:equationL}).
To distinguish different cases and conclude
we use a classical consequence of Lucas's theorem which says
that a binomial coefficient $\binom{a}{b}$ 
is divisible by $2$ if and only if at least one of the base $2$
digits of $b$ is greater than the corresponding digit of $a$.
\end{proof}

We use  the following representation lemma as a key point in the proof of Proposition \ref{covers}.
\begin{lemma}\label{lemma:representation}
Let $V$ be a  $\mathbb{F}_2$-vectorial subspace of 
$\mathbb{F}_q$ of codimension $2$. Then there exist two distincts elements $\alpha$ and 
$\alpha'$ in $\mathbb{F}_q^{\ast}$ such that $V=\Ima T_{\alpha,\alpha'}$ where
$T_{\alpha,\alpha'} (x) = x(x+\alpha)(x+\alpha')(x+\alpha+\alpha')$.
\end{lemma}
\begin{proof}
First we prove that $\Ima T_{\alpha,\alpha'} $ is the intersection 
of the kernels of the morphisms
$x \mapsto \Tr_{\mathbb{F}_{2^n} / \mathbb{F}_{2} } \left( \frac{x}{(\alpha^2+\alpha \alpha')^ 2 } \right) $
and 
$x \mapsto \Tr_{\mathbb{F}_{2^n} / \mathbb{F}_{2} } \left( \frac{x}{(\alpha'^2+\alpha \alpha')^2 }\right) $
where 
$\Tr_{\mathbb{F}_{2^n} / \mathbb{F}_{2} }$ is the Trace function 
relative to the extension $\mathbb{F}_{2^n} / \mathbb{F}_{2}$.
Let us prove 
 that $\Ima T_{\alpha,\alpha'}$
is included in the kernel of one the two morphisms. 
Indeed, 
if $z=T_{\alpha,\alpha'}(x)$, then
$z=u(u+ \gamma)$ with $\gamma=\alpha'^2+ \alpha \alpha'$
 and $u=x(x+\alpha)$. 
The Hilbert 90 Theorem implies that 
$\Tr_{\mathbb{F}_{2^n} / \mathbb{F}_{2} }( z  /\gamma^2)=0$ and we are done.
We have the inclusion in the kernel of the other morphism by symmetry, 
and we conclude with a dimension argument.

As any hyperplane of  $\mathbb{F}_{2^n}$ is the kernel
of a linear form $x \mapsto \Tr(w.x)$ for a good choice of $w \in \mathbb{F}^{\ast}_{2^n}$,
and as
$x\mapsto1/x^2$ is a bijection onto 
$\mathbb{F}^{\ast}_{2^n}$
it is now sufficient to prove that for all couple $(u,v)$ of
distinct elements of $\mathbb{F}_{2^n}^{\ast}$ there exists 
a couple of distinct elements  $(\alpha,\alpha')$ of
 $\mathbb{F}_{2^n}^{\ast}$ such that
$\alpha^2+\alpha \alpha'=u$ and $\alpha'^2+\alpha \alpha'=v$.
To this end, we consider the function
$
\Theta 
 \ : \ \mathbb{F}_{2^n}^{\ast} \times \mathbb{F}_{2^n}^{\ast} \setminus \Delta
\rightarrow \mathbb{F}_{2^n}^{\ast} \times\mathbb{F}_{2^n}^{\ast} \setminus \Delta$ which maps
 $(\alpha, \alpha')$ to $(\alpha^2+\alpha \alpha',\alpha'^2+\alpha \alpha')$
 where $\Delta$ denotes the diagonal.
It is well defined because if $\alpha^2+\alpha \alpha' = \alpha'^2+\alpha \alpha'$
then $\alpha^2=\alpha'^2$ and so  $\alpha=\alpha'$.
If $\Theta (\alpha_1,\alpha_1')=\Theta (\alpha_2,\alpha_2')$,
then one has the two equalities
$\alpha_1^2+\alpha_1 \alpha_1' = \alpha_2^2 + \alpha_2 \alpha_2'$ and
$\alpha_1'^2+\alpha_1 \alpha_1' = \alpha_2'^2 + \alpha_2 \alpha_2'$.
It implies $(\alpha_1+\alpha_1')^2=(\alpha_2+\alpha_2')^2$ and 
so there exists $\mu \in \mathbb{F}_{2^n}$ such that 
$\mu=\alpha_1+\alpha_1'=\alpha_2+\alpha_2'$.
Using the first equality one obtains $\alpha_1 \mu = \alpha_2 \mu$.
We know that $\mu \neq 0$,
otherwise  we would have $\alpha_1=\alpha_1'$, and 
$(\alpha_1,\alpha_1') \in \Delta$, a contradiction.
So we can deduce $\alpha_1 = \alpha_2$ 
and using the first equality one more time we have $\alpha_1 \alpha_1' = \alpha_2 \alpha_2'$,
and so $\alpha_1'=\alpha_2'$.
Hence the function  $\Theta$ is injective and thus bijective.
\end{proof}

\section{Main theorem}\label{section_main_theorem}

We will use all the previous propositions to prove  our main result, namely that
most polynomials $f$ over ${\mathbb F}_ {q}$ have
a maximal $\delta^2(f)$.
More precisely, we prove the following theorem.

\begin{theorem}\label{maintheorem}
Let $m$ be an integer such that $m\geqslant 7$ 
and $m\equiv 0 \pmod 8$ 
(respectively $m \equiv 1,2,7\pmod 8$),
 let $\delta_0=m-4$ (respectively  $\delta_0=m-5, 
m-6, m-3$). Then we have
$$\lim_{n\rightarrow \infty}
\frac{\sharp\{f\in{\mathbb F}_{2^n}[x] \mid \deg(f)=m,\ \delta^2(f)=\delta_0\}}
{\sharp\{f\in{\mathbb F}_{2^n}[x] \mid \deg(f)=m\}}=1.$$
\end{theorem}

\begin{proof}
Recall that we set $q=2^n$.
We fix an integer $m \geqslant 7$ and consequently 
an integer $d$ defined by Table \ref{table:definition_of_d}
and an integer $\tilde{d}$ depending only on $d$ as introduced in
Proposition \ref{proposition:bounded_nonMorse}.

Let $\varepsilon>0$.
We fix an integer $N_1$ satisfying the properties of Proposition \ref{application_Chebotarev}.
By Proposition \ref{covers} there exist integers $k$ and $N_2$ 
such that
for any $n \geqslant N_2$ we can choose  $k$ couples
 $(\alpha_1,\alpha_1'),\ldots,(\alpha_k,\alpha_k') $ of disctinct elements of 
$ {\mathbb F}_{q}^*$
such that
 for at least
  $(1-\varepsilon)(q-1)q^m-q^m$ polynomials $f\in{\mathbb F}_q[x]$
  of degree $m$ the polynomial
$L_{\alpha_i,\alpha_i'} (f)$ has degree $d$ for all $i$, and 
at least one of the $k$ equations 
$$\frac{b_1}{b_0} =x(x+\alpha_i)(x+\alpha_i')(x+\alpha_i+\alpha_i')$$
has a solution in ${\mathbb F}_q$, where 
$L_{\alpha_i,\alpha_i'}(f(x))= b_0 x^d + b_1 x^{d-1} + \cdots +b_d$. 
Finally, 
we fix an integer $N_3$ such that
for all $n\geqslant N_3$
\begin{equation}
\label{condition_un}
0 \leqslant \frac{q^m  + k \tilde{d}q^m}{(q-1)q^m}\leqslant \varepsilon.
\end{equation}
Let $n \geqslant \textrm{Max}(N_1,N_2,N_3)$ 
and  a polynomial $f$ associated to a couple $(\alpha_i, \alpha_i')$
satisfying the preceeding conditions.
If we suppose that $L_{\alpha_i,\alpha_i'}(f)$ is Morse,
then by  Proposition \ref{regular} 
the extension $\Omega/\mathbb{F}_q(t)$ is regular
where $\Omega$ is the Galois closure of $D_{\alpha_i,\alpha_i'}f (x)+t$.
Hence
by Proposition \ref{application_Chebotarev} there exists $\beta \in \mathbb{F}_q$ such that
$D^2_{\alpha_i , \alpha_i' } (f) (x) = \beta$ has $4d$ solutions in $\mathbb{F}_q$.
It amounts to saying that $\delta^2(f)=\delta_0$.
Let us count these polynomials: $f$ is choosen among 
the $(1-\varepsilon)(q-1)q^m-q^m$ polynomials given by Proposition \ref{covers},
but we have to remove the polynomials $f$ such that for all $i \in \{ 1, \ldots ,k \}$
the polynomial $L_{\alpha_i, \alpha_i'} (f)$ is non-Morse.
Thanks to Proposition \ref{proposition:bounded_nonMorse} 
we know we have to remove at most 
$k \tilde{d} q^m$ polynomials.
To obtain the density we have to divide by $(q-1)q^m$ which is the number of polynomials of degree $m$.
Finally, the condition (\ref{condition_un}) above ensures that this density is greater than or equal to
$1 - 2 \varepsilon$. 
\end{proof}

\section{The inversion mapping}\label{inverse}

We conclude the paper by the study of the second order differential uniformity of 
the inversion mapping from ${\mathbb F}_{q}$ (with $q=2^n$) to itself which sends $x$ to $x^{-1}$ if $x\not=0$ and 0 to 0  and which corresponds to the polynomial $f(x)=x^{q-2}$ of ${\mathbb F}_q[x]$. 
The $S$-box used by AES involves precisely this function in the case where  $n=8$.
Nyberg proved in \cite{Nyberg}  that it
 has a differential uniformity $\delta(f)=2$ for $n$ odd and $\delta(f)=4$ for $n$  even.
 
 We determine here its
second order differential uniformity over ${\mathbb F}_{2^n}$ for any $n$. By a direct computation, we can show that 
$\delta^2(f)=4$ over ${\mathbb F}_{2^n}$  for $n=2, 4$ and $5$
   and that  $\delta^2(f)=8$ for $n=3$. For $n\geqslant 6$, we have the following proposition.

\begin{proposition}
The inversion mapping $f$ over ${\mathbb F}_{2^n}$  has a second order differential uniformity $\delta^2(f)=8$ for any $n\geqslant 6$.
\end{proposition}

\begin{proof}
Set $q=2^n$ and let $\alpha, \alpha'\in{\mathbb F}_q^{\ast}$ such that $\alpha\not=\alpha'$ and $\beta\in{\mathbb F}_q$.
Consider the equation $D_{\alpha,\alpha'}f(x)=\beta$ i.e.
$$
x^{q-2}+(x+\alpha)^{q-2}+(x+\alpha')^{q-2}+(x+\alpha+\alpha')^{q-2}=\beta.
$$
Since $f$ is a monomial function, this equation can be written:
$$
\alpha'^{q-2}\left(\Bigl(\frac{x}{\alpha'}\Bigr)^{q-2}+\Bigl(\frac{x}{\alpha'}+\frac{\alpha}{\alpha'}\Bigr)^{q-2}+\Bigl(\frac{x}{\alpha'}+1\Bigr)^{q-2}+\Bigl(\frac{x}{\alpha'}+\frac{\alpha}{\alpha'}+1\Bigr)^{q-2}\right)=\beta.
$$

Thus in order to compute $\delta^2(f)$  we can suppose that $\alpha'=1$. So we consider now for $\alpha\in{\mathbb F}_q\setminus\{0,1\}$ and $\beta\in{\mathbb F}_q$ the number of solutions of the equation:

\begin{equation}
x^{q-2}+(x+\alpha)^{q-2}+(x+1)^{q-2}+(x+\alpha+1)^{q-2}=\beta.
\label{nouvelleequation}
\end{equation}

If $x\not\in\{0,1,\alpha,\alpha+1\}$, then this equation
 is equivalent to:
$$x^{-1}+(x+\alpha)^{-1}+(x+1)^{-1}+(x+\alpha+1)^{-1}=\beta$$
which is equivalent to:
\begin{equation}
\beta T_{\alpha,1}(x)+\alpha(\alpha+1)=0
\label{lequationreduite}
\end{equation}
where $T_{\alpha,\alpha'}(x)=x(x+\alpha)(x+\alpha')(x+\alpha+\alpha')$ as introduced in Section \ref{section:Geometric_and_arithmetic_mondormy_groups}.

Thus Equation \eqref{lequationreduite} has at most four solutions in ${\mathbb F}_q\setminus\{0,1, \alpha, \alpha+1\}$. Precisely, it has no solution or it has four solutions
since  $T_{\alpha,1}(x)=T_{\alpha,1}(x+\alpha)=T_{\alpha,}(x+1)=T_{\alpha,1}(x+\alpha+1)$.

An element $x\in\{0,1,\alpha,\alpha+1\}$
 is a solution of Equation \eqref{nouvelleequation} if and only if
 $\beta=\frac{\alpha^2+\alpha+1}{\alpha(\alpha+1)}.$
Now let us solve  Equation \eqref{nouvelleequation} in
${\mathbb F}_q\setminus\{0,1,\alpha,\alpha+1\}$ with such $\beta$.
 If $\beta=0$ then Equation  \eqref{lequationreduite} has no solution so we can suppose that $\beta\not=0$ i.e. $\alpha^2+\alpha+1\not=0$.
 Then  equation  \eqref{lequationreduite} can be written
$T_{\alpha,1}(x)=\gamma$
\label{derniereequation}
where
$\gamma=\frac{\alpha^2(\alpha^2+1)}{\alpha^2+\alpha+1}.$
We have shown in the proof of Lemma \ref{lemma:representation} that
 $\Ima T_{\alpha,\alpha'} $ is equal to the intersection 
of the kernels of the morphisms
$x \mapsto \Tr_{\mathbb{F}_{2^n} / \mathbb{F}_{2} } \left( \frac{x}{(\alpha^2+\alpha \alpha')^ 2 } \right)$
and 
$x \mapsto \Tr_{\mathbb{F}_{2^n} / \mathbb{F}_{2} } \left( \frac{x}{(\alpha'^2+\alpha \alpha')^2 }\right)$.
Hence the equation $T_{\alpha,1}(x)=\gamma$ has a solution if and only if
$\gamma$ is in the intersection of the kernels of these two maps, i.e.
\begin{equation}
\Tr_{\mathbb{F}_{2^n} / \mathbb{F}_{2} } \left( \frac{1}{\alpha^2+\alpha+1 } \right) =0\ \  {\rm and}\ \  \Tr_{\mathbb{F}_{2^n} / \mathbb{F}_{2} } \left( \frac{\alpha^2}{\alpha^2+\alpha+1 }\right) =0.
\label{system}
\end{equation}

\bigskip

In the case where $n$ is even,  any element  in the subfield ${\mathbb F}_{2^{n/2}}$ has a trace equal to zero. Thus, any $\alpha$ different from 0 and 1 in this subfield and with $\alpha^2+\alpha+1\not=0$  verifies the two previous conditions of \eqref{system}. Thus if the subfield  ${\mathbb F}_{2^{n/2}}$ have more than 4 elements, i.e. if $n>4$  then $\delta^2(f)=8$.

\bigskip

In order to solve the problem in the  case where 
$n$ is odd, consider the algebraic surfaces $S_1$ and $S_2$ in the affine space ${\mathbb A}^3$ given respectively by the equations
$(y^2+y)(x^2+x+1)=1$
and 
$(z^2+z)(x^2+x+1)=x^2.$
Consider the affine curve $C=S_1\cap S_2$ in ${\mathbb A}^3$.
By Hilbert 90 theorem, a solution $\alpha$ in ${\mathbb F}_{2^n}$ to Equations \eqref{system} corresponds to four 
points $(x,y,z)$ on $C$.

Furthermore if $(x,y,z)\in C$ then we can show that
$x(y^2+y) + y^2+y+z^2+z+1=0$
 and
$x(z^2+z+1) + y^2+y+1=0.$
Then we obtain:

\begin{equation}
(y^2+y)^2+(y^2+y)(z^2+z) +(z^2+z+1)^2=0.
\label{Klein}
\end{equation}
Consider the projection

$$
\begin{matrix}
\pi:&{\mathbb A}^3 & \longrightarrow & {\mathbb A}^2\\
& (x,y,z)& \longmapsto & (y,z)
\end{matrix}
$$
and the affine plane curve $D$ defined by  Equation (\ref{Klein}).
Consider also 
$Z=\{(y,z)\in{\mathbb A}^2 \mid y^2+y=0 \ \ {\rm and}\ \ z^2+z+1=0\}.$
The set $Z$ has 4 points and each of them has degree 2 over ${\mathbb F}_2$.
The projection $\pi$ provides an isomorphism between $C$ and $D\setminus Z$ whose inverse is given by:

\begin{align*}
D\setminus Z  & \to     C
\\
                (y,z)                          & \mapsto \begin{cases}
                                                         \left(\frac{z^2+z+1}{y^2+y}+1, y,z\right) & \text{if }y^2+y\neq 0,\\[0.5em]
                                                         \left(\frac{y^2+y+1}{z^2+z+1}, y,z\right) & \text{if }z^2+z+1\neq 0.\\
                                                         \end{cases}
\end{align*}

 Let us denote by $\overline{D}$ the projective closure of $D$ in the projective plane ${\mathbb P}^2$. It has 2 points at infinity and each of them has degree 2.
 
It follows that the curves $C$ and $\overline D$ have the same number of rational points over ${\mathbb F}_{2^n}$ for $n$ odd.
 Furthermore, 
 the curve $\overline D$ is a smooth projective plane quartic, so it is absolutely irreducible and has genus 3. By Serre-Weil theorem (see \cite{Serre_cras}),  the number of rational points over ${\mathbb F}_{2^n}$ of $\overline D$ verifies:
$$\sharp {\overline D}({\mathbb F}_{2^n})\geqslant 2^n+1-3[{2^{(n+2)/2}}].$$

So, if $n\geqslant 7$, we have 
$\sharp {C}({\mathbb F}_{2^n})\geqslant 63$
and then there are at least 15 solutions to Equations \eqref{system} and the result follows.
\end{proof}
\bigskip

\bigskip
\bigskip
\noindent
{\bf Acknowledgments:} The authors want to thank Felipe Voloch for lightning discussions,
particularly concerning the strategy described in Section \ref{section_good}. They also want  to thank Philippe Langevin, Ren\'e Schoof and David Kohel for a nice discussion concerning the last section and the referee for 
helpful comments.

\bigskip
\bigskip
\noindent
{\bf References:}

\bigskip

\bibliographystyle{plain}
\bibliography{biblio_yff}
\end{document}